\documentclass[11pt,a4paper,reqno]{amsart}
\usepackage{hyperref}
\usepackage[all]{xy}
\usepackage{amssymb}
\usepackage{amsmath}

\usepackage{color}
\usepackage[usernames,dvipsnames]{xcolor}

\font\cyr=wncyr10 scaled \magstep1%
\def\Sh{\text{\cyr Sh}}

\DeclareMathOperator{\cok}{coker}

\newcommand{\fp}{\mathfrak{p}}
\newcommand{\CL}{\mathcal{L}}

\newcommand{\Hom}{\text{Hom}}

\newcommand{\af}{\text{af}}
\newcommand{\Nr}{\text{Nr}}

\newcommand{\Sp}{\text{Spec} \,}

\newcommand{\ur}{{\mathrm{ur}}}

\chardef\bslash=`\\ 

\newtheorem{theorem}{Theorem}[section]

\newtheorem{prop}[theorem]{Proposition}
\newtheorem{lem}[theorem]{Lemma}
\newtheorem{cor}[theorem]{Corollary}

\theoremstyle{definition}

\newtheorem{remark}[theorem]{Remark}
\newtheorem{example}[theorem]{Example}

\numberwithin{equation}{section}

\newtheorem*{maintheorem*}{Main Theorem}
\theoremstyle{definition}
\newtheorem{definition}{Definition}

\newcommand{\CO}{\mathcal{O}}
\newcommand{\OV}{\un{\textbf{O}}_V}
\newcommand{\SOV}{\un{\textbf{SO}}_V}

\newcommand{\Oiy}{\CO_{\{\iy\}}}

\newcommand{\X}{\mathcal{X}}

\renewcommand{\sectionmark}[1]{}

\newcommand{\N}{\un{N}}
\newcommand{\sh}{\text{sh}}

\newcommand{\diag}{\text{diag}}

\newcommand{\Br}{{\mathrm{Br}}}

\newcommand{\iy}{\infty}

\newcommand{\bk}{\bigskip}
\newcommand{\fc}{\frac}
\newcommand{\G}{\Gamma}

\newcommand{\Pic}{\text{Pic~}}

\newcommand{\dl}{\delta}

\newcommand{\ov}{\overline}
\newcommand{\un}{\underline}

\newcommand{\BG}{\mathbb{G}}
\newcommand{\BF}{\mathbb{F}}

\renewcommand{\a}{\alpha}

\newcommand{\et}{\text{\'et}}
\newcommand{\p}{\varphi}

\newcommand{\Z}{\mathbb{Z}}
\newcommand{\Q}{\mathbb{Q}}
\newcommand{\A}{\mathbb{A}}

\begin{document}

\date{}


\baselineskip20pt
\setcounter{equation}{0}
\pagestyle{plain}
\pagenumbering{arabic}

\title{The Hasse principle for bilinear symmetric forms over a ring of integers of a global function field}

\author{Rony A. Bitan}

\begin{abstract}
Let $C$ be a smooth projective curve defined over the finite field $\BF_q$ ($q$ is odd)
and let $K=\BF_q(C)$ be its function field. 
Removing one closed point $C^\af = C-\{\iy\}$ results in an integral domain $\Oiy = \BF_q[C^\af]$ of $K$,   
over which we consider a non-degenerate bilinear and symmetric form $f$ with orthogonal group $\OV$. 
We show that the set $\text{Cl}_\iy(\OV)$ of $\Oiy$-isomorphism classes in the genus of $f$ of rank $n>2$,   
is bijective as a pointed set to the abelian groups $H^2_\et(\Oiy,\un{\mu}_2) \cong \Pic(C^\af)/2$, i.e. is an invariant of $C^\af$.   
We then deduce that any such $f$ of rank $n>2$ admits the local-global Hasse principal if and only if $|\Pic(C^\af)|$ is odd.  
For rank $2$ this principle holds if the integral closure of $\Oiy$ in the splitting field of $\OV \otimes_{\Oiy} K$ is a UFD.  
\end{abstract}

\maketitle{}

\markright{The Hasse principle in integral forms}

\pagestyle{headings}

\section{Introduction}
Let $C$ be a smooth, projective, geometrically connected curve     
defined over the finite field $\BF_q$ with $q$ odd,  
and let $K = \BF_q(C)$ be its function field. 
For any prime $\fp$ of $K$, let $v_\fp$ be the induced discrete valuation on $K$.   
We remove one closed point $\iy$ from $C$, resulting in an affine curve $C^\af$,  
and consider the following ring of $\{\iy\}$-integers of $K$:
$$ \Oiy = \BF_q[C^\af] := \{ a \in K \mid v_\fp(a) \geq 0 \ \ \forall \fp \neq \iy \}. $$
Throughout the paper, an \emph{integral form} on $V \cong \Oiy^n$    
refers to a bilinear and symmetric map $f:V \times V \to \Oiy$. 
It will be called \emph{unimodular} if it is non-degenerate at any closed point of $C^\af$,  
which is equivalent in this case to $\det(f) \in \BF_q^\times$.  
Two integral forms $f$ and $g$ on $V$ are $\Oiy$\emph{-isomorphic}   
if there exists $Q \in \textbf{GL}(V)$ such that $f(u,v) = g(Q u,Q v)$ for all $u,v \in V$. 

\bk

The standard approach to classifying bilinear forms over a global field such as $K$,   
basically relies on the Hasse-Minkowski principle which states that this classification, 
expressed by the first Galois cohomology set $H^1(K,\textbf{O}_V)$   
where $\textbf{O}_V$ stands for the orthogonal group of $f$,   
is equivalent to that obtained locally everywhere, 
namely, by $\prod_\fp H^1(\hat{K}_\fp,(\textbf{O}_V)_\fp)$   
where $\hat{K}_\fp$ is the complete localization of $K$ at a prime $\fp$  
and $(\textbf{O}_V)_\fp$ is the geometric fiber of $\textbf{O}_V$ there.   
However, if one considers the classification of integral forms, 
then this local-global principle fails, leading to the notion of a \emph{genus} of a form.  
In this paper, we aim to describe geometrically the violation of this principle.  
We express the classification of integral unimodular forms from the same genus 
via $H^1_\et(\Oiy,\SOV)$ (Proposition \ref{class number and cohomology}), 
where $\un{\textbf{SO}}_V$ is the special orthogonal group scheme of $f$ defined over $\Sp \Oiy$,     
and then show that this set is bijective as a pointed set for ranks $n > 2$ 
to the abelian group $H^2_\et(\Oiy,\un{\mu}_2)$, 
i.e. it is an invariant of $C^\af$ (Proposition \ref{H2 SO_V=1}).    
Furthermore, by proving that the Brauer group of the affine curve $C^\af$ is trivial (Lemma \ref{Br(Oiy)}),  
we conclude that $H^2_\et(\Oiy,\un{\mu}_2) \cong \Pic(C^\af)/2$ (Corollary \ref{H2 mu2}).    

\bk

This description leads us to assert the validity of the Hasse principle 
for unimodular integral forms of rank $n > 2$ if and only if $|\Pic(C^\af)|$ is odd.   
For $n=2$, the Hasse principle holds if the integral closure of $\Oiy$ 
in the splitting field of the generic fiber of $\OV$ is a UFD (Theorem \ref{main theorem}). 
This result can be considered as a generalization of Theorem~3.1 in \cite{Ger1} 
in which the elementary case of $\Oiy = \BF_q[t]$ is treated (Example \ref{over affine line}). 
Its proof was initially based on the reduction by Harder, 
of the unimodular theory over $\BF_q[t]$,  
to the theory of spaces over $\BF_q$ (see \cite[Theorem~7.13]{Ger2}).

\bk \noindent

{\bf Acknowledgements:} 
The author thanks B.~Conrad, L.~Gerstein, P.~Gille and B.~Kunyavski\u\i \ 
for valuable discussions concerning the topics of the present article, 
and the anonymous referee for his instructive remarks.  

\bk

\section{Classification over rings of integers} 
The geometrically connected projective curve $C$ remains geometrically connected 
after removing the closed point $\iy$, resulting in $C^\af$. 
In order to classify integral forms,  
we shall refer to the fundamental group $\pi_1(C^\af,a)$ of $C^\af$ w.r.t. some geometric base point $a$,     
as defined by Grothendieck in \cite[V,\S4~and~\S7]{SGA1}.  
Up to isomorphism, this group (as a topological group)  
does not depend on the choice of the base point (see \cite[Ch.I,~Remark~5.1]{Mil}).  
Therefore, where one is only concerned with the group-theoretic structure of $\pi_1(C^\af,a)$, 
we may omit the base-point and write just $\pi_1(C^\af)$. 

\bk

For any prime $\fp$ of $K$, let $\CO_\fp$ be the discrete valuation ring of $K$ w.r.t. to $v_\fp$  
and let $K_\fp$ be its fraction field.   
Let $\hat{K}_\fp$ be the completion of $K_\fp$ and let $\hat{\CO}_\fp$ be its ring of integers. 
Let $k_\fp = \hat{\CO}_\fp / \fp$ be the corresponding (finite) residue field. 
Let $\hat{K}_\fp^\ur$ be the maximal unramified extension of $\hat{K}_\fp$ 
and let $\hat{\CO}_\fp^\sh$ be its ring of integers. 
Given a smooth group scheme $\un{G}_\fp$ defined over $\Sp \hat{\CO}_\fp$,  
the set $H^1_\et(\hat{\CO}_\fp,\un{G}_\fp)$ 
is bijective to the Galois cohomology set $H^1(\hat{\CO}_\fp,\un{G}_\fp)$, 
while the Galois group taken under consideration is 
$\mathrm{Gal}(\hat{\CO}_\fp^\sh / \hat{\CO}_\fp) = \mathrm{Gal}(\ov{k}_\fp / k_\fp)$   
where $\ov{k}_\fp$ stands for the algebraic closure of $k_\fp$.   
For a smooth group scheme $\un{G}$ defined over $\Oiy$, 
by writing $H^1_\et(\Oiy,\un{G}) \cong H^1_{\text{flat}}(\Oiy,\un{G})$ 
we shall refer to the action of the aforementioned (total) fundamental group $\pi_1(C^\af)$ 
(see \cite[VIII~Corollaire~2.3]{SGA4} for the \'etale, 
flat and Galois cohomology sets bijections in the smooth case).  
 
\bk

\section{Integral schemes and \'etale cohomology} \label{class group}
Let $\un{G}$ be an affine, flat and smooth 
group scheme defined over $\Sp \Oiy$ with generic fiber $G$.  
For any prime $\fp$ of $K$, the localization $(\Oiy)_\fp$ is a base change of $\CO_\fp$. 
Thus the bijection $\Sp(\Oiy)_\fp \to \Sp \CO_\fp$ is faithfully flat (see \cite[Theorem~3.16]{Liu})   
and so $\un{G}$, extended to be defined over $\Sp \hat{\CO}_\fp$ and denoted by $\un{G}_\fp$, is also smooth. 
Under these settings, we shall refer to the adelic group $\un{G}(\A)$  
and to its subgroup over the ring of $\{\iy\}$-integral ad\`eles    
$\A_\iy := \hat{K}_\iy \times \prod_{\fp \neq \iy} \hat{\CO}_\fp$.  

\begin{definition}
The \emph{class set} of $\un{G}$ is the set of double cosets  
$\mathrm{Cl}_\iy(\un{G}) := \un{G}(\A_\iy) \backslash \un{G}(\A) / G(K)$.  
It is finite (cf. \cite[Thm~1.3.1]{Con1}).  
Its cardinality, denoted by $h_\iy(\un{G})$, is called the \emph{class number} of $\un{G}$.  
\end{definition}

\begin{theorem} \label{Nis} \rm{(Ye. Nisnevich, \cite[Theorem~I.3.5]{Nis})}.   
There is an exact sequence of pointed sets: 
$$
1 \to \mathrm{Cl}_\iy(\un{G}) \to H^1_\et(\Oiy,\un{G}) \to H^1(K,G) \times \prod_{\fp \neq \iy} H^1(\hat{\CO}_\fp,\un{G}_\fp).   
$$           
\end{theorem}

\begin{lem} \label{H1(Gsc) trivial2} 
Suppose $\un{G}$ (being affine, flat and smooth defined over $\Sp \Oiy$) is connected,  
and that $G$ is almost simple, simply connected and $\hat{K}_\iy$-isotropic. 
Then $H^1_\et(\Oiy,\un{G})=0$. 
\end{lem}

\begin{proof} 
At any prime $\fp$, as $\hat{\CO}_\fp$ is Henselian, we have 
$H^i(\hat{\CO}_\fp,\un{G}_\fp) \cong H^i(k_\fp,\ov{G}_\fp)$ for $i \geq 0$ 
where $\ov{G}_\fp := \un{G}_\fp \otimes_{\Sp \hat{\CO}_\fp} k_\fp$ 
(see Remark 3.11(a) in \cite[Ch.~III,~\S3]{Mil}).   
The right set for $i=1$ is trivial by Lang's Theorem 
(see \cite[Ch.VI, Prop.~5]{Ser}).   
Furthermore, $H^1(K,G)$ is trivial in the simply connected case due to Harder's result (see \cite[Satz~A]{Hard}), 
and so Nisnevich's sequence from Theorem \ref{Nis} obtained for $\un{G}$, 
shows that $\mathrm{Cl}_\iy(\un{G})$ is bijective to $H^1_\et(\Oiy,\un{G})$. 
But as $G$ is almost simple, simply connected and $\hat{K}_\iy$-isotropic, 
it admits the strong approximation property w.r.t. $S=\{\iy\}$ (see \cite[Theorem~A]{Pra}), 
which means that $\mathrm{Cl}_\iy(\un{G})$ is trivial, and the assertion follows. 
\end{proof}

\begin{lem} \label{Br(Oiy)} 
$\Br(\Oiy)=1$. 
\end{lem}

\begin{proof}
As $C^\af$ is smooth, $H^2_\et(\Oiy,\un{\BG}_m) = \Br(\Oiy)$     
classifying Azumaya $\Oiy$-algebras (see \cite[\S 2]{Mil}). 
Let $[D]$ be a class of central simple algebras in $\Br(K)$.  
At any prime $\fp$, $[D]$ is associated by the residue map $r_\fp$ with an extension of $k_\fp$, 
representing thus a class in $H^1(k_\fp,\Q/\Z)=\Hom(k_\fp,\Q/\Z)$ (the Galois action is trivial). 
The latter term is isomorphic to $\Q/\Z$, 
as the absolute Galois group of any finite field $k_\fp$ is isomorphic to $\hat{\Z}$.    
The ramification map $a:= \oplus_\fp r_\fp$ 
yields then the exact sequence from Class Field Theory (see Theorem 6.5.1 in \cite{GS}):   
\begin{equation} \label{sum Hasses invariants}
1 \to \Br(K) \xrightarrow{a = \oplus_\fp r_\fp} \bigoplus\limits_\fp \Q / \Z \xrightarrow{\sum_\fp \text{Cor}_\fp} \Q/\Z \to 1   
\end{equation} 
in which the corestriction map $\text{Cor}_\fp$ for any $\fp$ is an isomorphism 
induced by the Hasse-invariant $\Br(\hat{K}_\fp) \cong \Q/\Z$ (cf. \cite[Proposition~6.3.9]{GS}).  
On the other hand, as all residue fields of $K$ are finite thus perfect, 
and $C^\af$ is a one-dimensional regular scheme over $\BF_q$, 
it admits due to Grothendieck the following exact sequence 
(see \cite[Proposition~2.1]{Gro} and \cite[Example~2.22, case (a)]{Mil}):
\begin{equation}
1 \to \Br(\BF_q[C^\af]=\Oiy) \to \Br(\BF_q(C^\af) = K) \xrightarrow{\oplus_{\fp \neq \iy} r_\fp} \bigoplus\limits_{\fp \neq \iy} \Q/\Z,   
\end{equation} 
which means that $\Br(\Oiy)$ is the subgroup of $\Br(K)$ of classes that vanish under $r_\fp$ at any $\fp \neq \iy$.  
Thus omitting these $r_\fp, \fp \neq \iy$ in sequence (\ref{sum Hasses invariants}), 
results in $\Br(\Oiy) = \ker [\Q/\Z \xrightarrow{\text{Cor}_\iy} \Q/\Z]$=1. 
\end{proof}

\begin{cor} \label{H2 mu2} 
There is an isomorphism of abelian groups: $\Pic(C^\af)/2 \cong H^2_\et(\Oiy,\un{\mu}_2)$. 
\end{cor}

\begin{proof}
\'Etale cohomology applied on the Kummer's exact sequence defined over $\Sp \Oiy$:
\begin{equation*} 
1 \to \un{\mu}_2 \to \un{\BG}_m \to \un{\BG}_m \to 1 
\end{equation*}
gives rise to the following long exact sequence:  
\begin{align*}
           \Pic(C^\af) \xrightarrow{\phi:[\CL] \mapsto [2\CL]} \Pic(C^\af) 
           \to H^2_\et(\Oiy,\un{\mu}_2) \to \Br(\Oiy) \stackrel{(\ref{Br(Oiy)})}{=} 1  
\end{align*}
in which $\Pic(C^\af)/2 = \cok(\phi) \cong H^2_\et(\Oiy,\un{\mu}_2)$. 
\end{proof}

\begin{definition} \label{def torus}
Let $S$ be a scheme and $G$ an $S$-group. 
$G$ is an \emph{$S$-torus of rank $r$} if it is locally isomorphic in the $\text{fpqc}$-topology 
to $\un{\BG}^r_m$ (cf. \cite[Exp.~IX,~Def.~1.3]{SGA3}).   
\end{definition}

\begin{lem} \label{etale isotrivial}
Let $h: S' \to S$ be a finite surjective morphism of integral schemes, where $S$ is Noetherian, normal and one dimensional. 
Then $\N:=R^{(1)}_{S'/S}(\un{\BG}_m)$ is an $S$-torus iff $h$ is \'etale.   
\end{lem}

\begin{proof} 
Under the Lemma's hypothesis on $h$, it is locally free (cf. \cite[Lemma 5.2.4]{Sza}).  
Hence the induced Weil restriction functor $\un{R}:=R_{S'/S}(\un{\BG}_m)$ exists 
(cf. \cite[\S 7.6~Theorem~4]{BLR}),   
and so $\N$, being equal to $\ker[\un{R} \stackrel{\det}{\longrightarrow} \un{\BG}_{m}]$ 
(see Ex.\P 9(c),~p.148 \cite{Bou}), is well defined. 

$S'$ is integral thus connected. 
So given that $h$ is \'etale over the normal scheme $S$, 
it admits a Galois closure $S'' \to S$ that factors through it 
(see \cite[Proposition~5.3.9]{Sza} and \cite[Theorem~4']{BS}).  
Thus the associated fundamental group $\G$ 
admits a subgroup $\G_0 :=\text{Aut}(S''|S') \subset \G$ 
consisting of automorphisms of $S''$ that fix $S'$,  
such that: $S' \otimes_S S'' \cong (S'')^{|\G/\G_0|}$ (see \cite[Proposition~5.3.8]{Sza}). 
Therefore: $ \un{R} \otimes_S S'' \cong \BG_{m,S''}^{|\G/\G_0|}$ 
and so: 
$$ 
\N \otimes_S S'' \cong \ker\left[\un{\BG}_{m,S''}^{|\G/\G_0|} \stackrel{\det}{\longrightarrow} \un{\BG}_{m,S''}\right] \cong \un{\BG}^{{|\G/\G_0|}-1}_{m,S''},
$$ 
i.e. $\N$ is an $S$-torus.  

Conversely, if $h$ ramifies at some prime, then $\N$ is not reductive there (see \cite[\S 10.5]{Vos}). 
So given that $h$ is locally free thus flat, it must be \'etale as well.  
\end{proof}

\bk

\section{The Hasse principle and the class group of the orthogonal group} \label{Hasse Section}
Let $\X$ be the scheme of invertible symmetric $n \times n$-matrices with entries in $\Oiy$. 
It is a $\Sp \Oiy$-scheme, and its points correspond to (non-degenerate) $n$-dimensional integral forms, 
on which $\un{\textbf{GL}}_{n}$ defined over $\Sp \Oiy$ acts by 
$$ \forall g \in \un{\textbf{GL}}_n, \ F \in \X \ : \ g*F = g^t F g. $$
Let $f$ be an integral unimodular form represented by $F \in \X$. 
Then the \emph{orthogonal group} $\un{\textbf{O}}_V$ associated to $(V,f)$ is the stabilizer of $F$.  
Since this action is defined over $\Oiy$, it is an affine scheme defined over $\Sp \Oiy$. 
Its generic fiber is $\textbf{O}_V := \un{\textbf{O}}_V \otimes_{\Sp \Oiy} K$. 
As $2$ is a unit in $\Oiy$, 
the \emph{special orthogonal group} $\un{\textbf{SO}}_V$ is $\ker[\det: \OV \to \un{\mu}_2]$,  
where $\un{\mu}_2 := \Sp \Oiy[x]/(x^2-1)$  
(see 
Definition 1.6 and Corollary 2.5 in \cite{Con2}).    
For the same reason ($2$ is a unit in $\Oiy$), 
$\OV$ is smooth regardless of the parity of $n$, 
and $\SOV$ is a smooth closed subgroup with connected fibers (\cite[Theorem~1.7]{Con2}). 
If $n$ is even, then $\un{\textbf{O}}_V$ is a semidirect product of $\un{\textbf{SO}}_V$ and $\un{\mu}_2$ (see Corollary 2.5 and Remark 2.6 in \cite{Con2}), 
and it is a direct product of these if $n$ is odd (cf. \cite[Proposition~3.4]{Con2}).

\begin{definition}
Two integral forms share the same \emph{genus} 
if they are isomorphic over $K$ and over $\hat{\CO}_\fp$ for all primes $\fp$. 
We denote by $\text{gen}(f)$ the set of all integral forms of the same genus as $f$. 
\end{definition}

\begin{definition}
Given an integral form $f$, let $c(f)$ denote the number of $\Oiy$-isomorphism classes in $\text{gen}(f)$. 
We say that the \emph{Hasse principle} holds for $f$ if $c(f)=1$.  
\end{definition}

Platonov and Rapinchuk have shown in \cite[Prop.~8.4]{PR} -- in the number field case --    
that $c(f)$ equals the class number of its orthogonal group. 
In the following, we shall sketch briefly their proof, this time in the function field case: 

We consider the above $\Oiy$-scheme $\un{\textbf{GL}}_n$ (in which $\un{\textbf{O}}_V$ is embedded),  
its subgroup $\un{\textbf{SL}}_n$ and their extensions defined over $\hat{\CO}_\fp$ at any prime $\fp$ 
(see Section \ref{class group}) while referring to their adelic groups.  
Any element of $\un{\textbf{O}}_V(\A)$ can be put in $\un{\textbf{SL}}_n(\A)$ 
by multiplying by a suitable element of $\un{\textbf{GL}}_n(\A_\iy)$. 
Since the $K$-group $\textbf{SL}_n$ is split, simple and simply connected,  
it admits the strong approximation property whence 
$\un{\textbf{SL}}_n(\A) = \un{\textbf{SL}}_n(\A_\iy) \textbf{SL}_n(K)$ (cf. {\cite[Theorem~A]{Pra}}).  
It follows that $\un{\textbf{O}}_V(\A) \subseteq \un{\textbf{GL}}_n(\A_\iy) \textbf{GL}_n(K)$. 
Now according to the Stabilizer Formula \cite[Theorem~8.2]{PR}, 
$c(f)$ is equal to the number of double cosets $\un{\textbf{O}}_V(\A_\iy) \cdot x \cdot \un{\textbf{O}}_V(K)$ 
in the principal coset $\un{\textbf{GL}}_n(\A_\iy) \textbf{GL}_n(K)$ which is $h_\iy(\un{\textbf{O}}_V)$. 

\begin{cor} \label{c(f)}
$c(f) = h_\iy(\un{\textbf{O}}_V)$. 
\end{cor}      

\begin{prop} \label{class number and cohomology} 
There is a bijection of finite pointed sets: 
$\mathrm{Cl}_\iy(\un{\textbf{O}}_V) \cong H_\et^1(\Oiy,\un{\textbf{SO}}_V)$. 
\end{prop}

\begin{proof} 
Being affine, flat and smooth, $\un{\textbf{O}}_V$ admits by the Nisnevich's Theorem \ref{Nis} the exact sequence of pointed sets:
\begin{equation} \label{Nis OV}
1 \to \mathrm{Cl}_\iy(\un{\textbf{O}}_V) \to H^1_\et(\Oiy,\un{\textbf{O}}_V) \to H^1(K,\textbf{O}_V) 
                                         \times \prod_{\fp \neq \iy} H^1(\hat{\CO}_\fp,(\un{\textbf{O}}_V)_\fp).  
\end{equation}                                                                                                                                
Let $W(*)$ denote the Witt ring for the ring $*$. 
By Witt's Theorem, two forms are isomorphic 
if and only if they belong to the same Witt class and have the same rank (see \cite[Cor.~3.3]{MH}), 
whence $H^1_\et(\hat{\CO}_\fp,(\un{\textbf{O}}_V)_\fp)$ injects into $W(\hat{\CO}_\fp)$ 
and $H^1(\hat{K}_\fp,(\un{\textbf{O}}_V)_\fp)$ into $W(\hat{K}_\fp)$.  
Since $\hat{K}_\fp$ is complete, $W(\hat{\CO}_\fp) = W(k_\fp)$ injects into $W(\hat{K}_\fp)$ 
and we obtain the following commutative diagram:
$$\xymatrix{
H^1_\et(\hat{\CO}_\fp,(\un{\textbf{O}}_V)_\fp) \ar[r] \ar@{^{(}->}[d] & H^1(\hat{K}_\fp,(\un{\textbf{O}}_V)_\fp) \ar@{^{(}->}[d]  \\
W(\hat{\CO}_\fp)              \ar@{^{(}->}[r] & W(\hat{K}_\fp) 
}$$
which shows that $H^1(\hat{\CO}_\fp,(\un{\textbf{O}}_V)_\fp)$ embeds into 
$H^1(\hat{K}_\fp,(\un{\textbf{O}}_V)_\fp)$ for any $\fp$. 
Then due to Corollary 3.6 in \cite{Nis}, sequence (\ref{Nis OV}) simplifies to: 
\begin{equation}\label{Nis for O_V}
1 \to \mathrm{Cl}_\iy(\un{\textbf{O}}_V) \to H_\et^1(\Oiy,\un{\textbf{O}}_V) \to H^1(K,\textbf{O}_V). 
\end{equation}
As $C^\af$ is assumed to be smooth, $\Sp \Oiy = \Sp \BF_q[C^\af]$ is a normal scheme, 
i.e. it is integrally closed locally everywhere.   
In this case any finite \'etale covering of $C^\af$ arises by the normalization of $\Sp \Oiy$ 
in some separable unramified extension of $K$ (see \cite[Theorem~6.13]{Len}).    
Thus any non-trivial 1-cocycle in $H^1_\et(\Oiy,\un{\mu}_2)$ 
remains non-trivial after tensoring with $K$,  
i.e. $H^1_\et(\Oiy,\un{\mu}_2)$ is embedded in its generic fiber.   
Hence, $\un{\mu}_2$ being a direct or semidirect summand in $\un{\textbf{O}}_V$ 
(see at the beginning of this section), 
can be canceled in sequence (\ref{Nis for O_V}), leading to:  
$$ \mathrm{Cl}_\iy(\un{\textbf{O}}_V) \cong \ker[H^1_\et(\Oiy,\SOV) \to H^1(K,\textbf{SO}_V)]. $$ 
The situation for $\SOV$ is simpler: 
having (smooth) connected fibers, $H^1_\et(\hat{\CO}_\fp,(\SOV)_\fp)$ 
vanishes for all primes $\fp$ (by Lang's Lemma), 
thus not only admitting again due to Corollary 3.6 in \cite{Nis} the exact sequence: 
\begin{equation} \label{SOV sequence}
1 \to \mathrm{Cl}_\iy(\un{\textbf{SO}}_V) \to H_\et^1(\Oiy,\un{\textbf{SO}}_V) \xrightarrow{\p} H^1(K,\textbf{SO}_V), 
\end{equation}
which shows that $\mathrm{Cl}_\iy(\un{\textbf{O}}_V) = \mathrm{Cl}_\iy(\SOV)$,    
but can be simplfied even more to (cf. \cite[I.3.5.2]{Nis} and \cite[Theorem~3.4]{Gon1}): 
$$ 1 \to \mathrm{Cl}_\iy(\un{\textbf{SO}}_V) \to H_\et^1(\Oiy,\un{\textbf{SO}}_V) \to \Sh^1_{\{\iy\}}(K,\textbf{SO}_V) \to 1 $$
in which the right term is the first Tate-Shafarevich group w.r.t. $\{\iy\}$, namely: 
$$ 
\Sh^1_{\{\iy\}}(K,\textbf{SO}_V) := \ker\left[ H^1(K,\textbf{SO}_V) \to \prod_{\fp \neq \iy} H^1(\hat{K}_\fp,(\textbf{SO}_V)_\fp) \right]. 
$$  
The pointed set $H^1(K,\textbf{SO}_V)$ properly (i.e. by $\det=1$ isomorphisms) 
classifies $K$-forms isomorphic to $f$ over some finite Galois extensions of $K$, 
therefore sharing all the same rank and discriminant. 
So according to the Hasse-Minkowsky principle (cf. \cite[VI.3.1]{Lam}),  
these forms are classified via their Hasse-invariant locally everywhere.  
But as the base point $f$ is unimodular, representatives of any class in $H^1(K,\textbf{SO}_V)$ are $\Oiy$-regular as well,  
thus their local Hasse-invariants belong to $\Br(\Oiy)$, being trivial by Lemma \ref{Br(Oiy)}. 
This means that $\text{Cl}_\iy(\SOV)$ surjects on $H_\et^1(\Oiy,\un{\textbf{SO}}_V)$. 
On the other hand, $\text{Cl}_\iy(\SOV)$ is bijective to the first Nisnevich's cohomology set $H^1_{\text{Nis}}(\Oiy,\SOV)$ 
(cf. \cite[I.~Theorem~2.8]{Nis} and \cite[4.1]{Mor}), 
classifying $\SOV$-torsors in the Nisnevich's topology. 
But Nisnevich's covers are \'etale,  
so it is a subset of $H_\et^1(\Oiy,\un{\textbf{SO}}_V)$, and the assertion follows.  
\end{proof}

In particular, Proposition \ref{class number and cohomology} plus Corollary \ref{c(f)} yield:

\begin{cor} \label{Hasse}
The Hasse principle holds for an integral unimodular form 
iff $H^1_\et(\Oiy,\un{\textbf{SO}}_V) = 0$.    
\end{cor}

For rank $n > 2$, we consider the following construction (see in \cite[\S2]{Bas}):  \\
Let $\textbf{C}(f)$ be the Clifford algebra associated to $f$.  
It is a $\Z_2$-graded algebra. 
The linear map $v \mapsto -v$  on $V$   
extends to an algebra automorphism $\a: \textbf{C}(f) \to \textbf{C}(f)$ 
(acting as the identity on the even part and negation on the odd part).  
The \emph{Clifford group} associated to $(V,f)$ is 
$$ 
\textbf{CL}(f) := \{ u \in \textbf{C}(f)^\times \ : \  \a(u)vu^{-1} \in V \ \forall v \in V \}. 
$$
The identity map on $V$ (viewed as its inclusion in the opposite algebra of $\textbf{C}(f)$) 
extends to an anti-automorphism of $\textbf{C}(f)$ which we denote by $t$. 
The composition $\a \circ t$ mapping $v \mapsto \bar{v}$    
gives rise to the norm $N:\textbf{CL}(f) \to \Oiy^\times = \BF_q^\times: v \mapsto v\bar{v}$ 
(for $v \in V$ it is just $N(v) = -v^2=-f(v,v)$).   
We define $\un{\textbf{Pin}}_V(\Oiy) := \ker(N)$. 
This group admits an underlying group scheme over $\Sp \Oiy$ 
which we denote by $\un{\textbf{Pin}}_V$. 
The homomorphism $\pi:\un{\textbf{Pin}}_V \to \un{\textbf{O}}_V$ 
sending $v$ to the isometry stabilizing it, is a double covering,  
yielding the following short exact sequence of $\Oiy$-group schemes: 
\begin{equation} \label{Spin short sequence}
1 \to \un{\mu}_2 \to \un{\textbf{Spin}}_V \stackrel{\pi}{\rightarrow} \un{\textbf{SO}}_V \to 1 
\end{equation}
where $\un{\textbf{Spin}}_V := \pi^{-1}(\un{\textbf{SO}}_V) \subset \un{\textbf{Pin}}_V$.  

\begin{prop} \label{H2 SO_V=1} 
Let $(V,f)$ be an integral unimodular space of rank $n>2$. 
Then $\text{Cl}_\iy(\OV)$ is bijective as a pointed set to $H^2_\et(\Oiy,\un{\mu}_2)$ 
(being isomorphic to $\Pic(C^\af)/2$).     
\end{prop}

\begin{proof} 
The schemes in sequence (\ref{Spin short sequence}) are smooth, whence \'etale cohomology yields the exact sequence 
of pointed sets: 
\begin{equation} 
H^1_\et(\Oiy,\un{\textbf{Spin}}_V) \to H^1_\et(\Oiy,\un{\textbf{SO}}_V) \stackrel{\dl}{\rightarrow} H^2_\et(\Oiy,\un{\mu}_2) 
\end{equation}
in which since $\Oiy$ is of Douai type -- see Definition~5.2 and Example~5.4(iii) in \cite{Gon2} -- 
and as $\un{\textbf{SO}}_V = \un{\textbf{Spin}}_V^{\text{ad}}$ while: $Z(\un{\textbf{Spin}}_V) = \un{\mu}_2$,  
$\dl$ is surjective. 
Furthermore, $\un{\textbf{Spin}}_V$ is affine, smooth and connected, 
and its generic fiber is simple, simply connected. 
As $\det(f) \in \BF_q^\times$, the generic fiber of $(V,f)$ 
admits a regular model over $\hat{\CO}_\iy$ (see \cite[92:1]{OMe}). 
Its reduction at $\iy$ remains regular of dimension $n \geq 3$ over the finite field $k_\iy$ thus isotropic (\cite[62:1b]{OMe}). 
Then its lift back to $\hat{\CO}_\iy$ is again isotropic due to Hensel's Lemma (see \cite[III.~Lemma~19.4]{EKM}),  
as well as $\textbf{Spin}_V$ over $\hat{K}_\iy$. 
Hence $H^1_\et(\Oiy,\un{\textbf{Spin}}_V)$ is trivial by Lemma \ref{H1(Gsc) trivial2},  
which means due to the exactness that $\ker(\dl) = \{0\}$.  
This does not imply yet that $\dl$ is injective, since $\un{\textbf{SO}}_V$ is non-commutative for $n>2$,  
whence $H^1_\et(\Oiy,\un{\textbf{SO}}_V)$ has no reason to be a group. 
In order to deduce the injectivity of $\dl$, we consider the following diagram 
induced by some non-trivial $\un{\textbf{SO}}_V$-torsor $P$, 
as described in \cite[Cha.~IV, Proposition~4.3.4]{Gir}:  
$$\xymatrix{
H^1_\et(\Oiy,\un{\textbf{SO}}_V)    \ar[r]^{\dl} \ar[d]_{\cong}^{\theta_P} & H^2_\et(\Oiy,\un{\mu}_2) \ar[d]^{r}_{\cong}   \\
H^1_\et(\Oiy,^P \un{\textbf{SO}}_V) \ar[r]^{\dl'}                  & H^2_\et(\Oiy,\un{\mu}_2) 
}$$
in which the map $\dl'$ is the one obtained by applying \'etale cohomology on 
the short exact sequence (\ref{Spin short sequence})  
while replacing $\SOV$ by the twisted group scheme $^P\SOV = \un{\textbf{SO}}_{(^P V)}$,     
$\theta_P$ is the induced twisting bijection, 
and $r$ is the translation by $-\dl(P)$. 
According to \cite[Cha.~IV, Proposition~4.3.4(i),(ii)]{Gir} 
this diagram is commutative and there is a bijection: 
$$ \left\{ x \in H^1_\et(\Oiy,\un{\textbf{SO}}_V) : \dl(x) = \dl(P) \right\} \cong \ker(\dl). $$
But as shown above, in our case $\ker(\dl)$ is trivial, 
implying that $\dl$ is injective and eventually is a bijection.  
Due to Proposition \ref{class number and cohomology}, 
$\text{Cl}_\iy(\OV)$ is bijective as a pointed set to $H^1_\et(\Oiy,\un{\textbf{SO}}_V)$, 
and therefore is bijective as a pointed set to $H^2_\et(\Oiy,\un{\mu}_2)$ as well. 
The rest is Corollary \ref{H2 mu2}.     
\end{proof}

\begin{theorem} \label{main theorem} 
Let $f$ be a unimodular form of rank $n$ defined over $\BF_q[C^\af]$. \\
The Hasse principle holds for $f$: \\ 
for $n = 2$ if the integral closure of $\BF_q[C^\af]$ in the splitting field of $\textbf{SO}_V$ is a UFD, and: \\
for $n > 2$ if and only if $|\Pic(C^\af)|$ is odd. 
\end{theorem}

\begin{proof} 
For rank $2$, $\SOV$ is a one dimensional norm $\Oiy$-torus.  
This derives from being one dimensional and smooth, and from the connectivity of the fibers (see at the beginning of this section). 
According to Lemma \ref{etale isotrivial}, such one dimensional norm $\Oiy$-tori arise from quadratic \'etale extensions, 
hence are being classified by $H^1_\et(\Oiy,\Z/2\Z)$.  
If $\Oiy$ is a UFD, 
the Kummer sequence defined over $\Sp \Oiy$ implies that ($2 \in \Oiy^\times$, hence the scheme $\Z/2\Z$ is isomorphic to $\un{\mu}_2$  over $\Sp \Oiy$):
$$ H^1_\et(\Oiy,\Z/2\Z) \cong \Oiy^\times/ (\Oiy^\times)^2 = \BF_q^\times/ (\BF_q^\times)^2 \cong H^1_\et(\BF_q,\Z/2\Z). $$
This means that given that $\Oiy$ is a UFD, any quadratic extension of it, producing a one dimensional norm $\Oiy$-torus,  
arises from a quadratic extension of $\BF_q$ (recall that $\text{char}(K) \neq 2$, 
hence the quadratic Artin-Schreier extensions are not to be considered here). 
Now if $\SOV$ splits over $\Oiy$, then $H^1_\et(\Oiy,\SOV) = H^1_\et(\Oiy,\un{\BG}_m) = \Pic(C^\af) = 0$ as $\Oiy$ is a UFD. 
Otherwise, it fits into an exact sequence of $\Oiy$-tori:
\begin{equation} \label{norm torus sequence}
1 \to \SOV \to \un{R}:=R_{\Oiy'/\Oiy}(\un{\BG}_m) \to \un{\BG}_m \to 1  
\end{equation} 
in which $\Oiy'$ is assumed to be a UFD. 
As $\ov{\textbf{SO}}_V := \SOV \otimes_{\Sp \Oiy} \BF_q$ is connected, by Lang's Lemma $H^1(\BF_q,\ov{\textbf{SO}}_V) \cong \BF_q^\times / \Nr(\BF_{q^2}^\times)=1$. 
But $\BF_{q^2}^\times \subseteq \Oiy'^\times = \un{R}(\Oiy)$, which means that $\un{R}(\Oiy) \to \BF_q^\times$ is surjective.  
Moreover, by Shapiro's Lemma $H^1_\et(\Oiy,\un{R}) \cong H^1_\et(\Oiy',\un{\BG}_{m,\Oiy'}) = \Pic(\Oiy')$ being trivial by the assumption. 
Thus applying \'etale cohomology on sequence (\ref{norm torus sequence}) implies that $H^1_\et(\Oiy,\SOV)$ vanishes as well, 
and the assertion follows from Corollary \ref{Hasse}. 

For higher ranks, this is just Proposition \ref{H2 SO_V=1} plus Corollary \ref{H2 mu2}. 
\end{proof}

\begin{example} \label{over affine line}
Let $C$ be of genus zero having a $\BF_q$-rational point which we assign as $\iy$. 
Then $\Oiy = \BF_q[C^\af]$ is a UFD, as well as any scalar extension of it (see \cite[Theorem~5.1]{Sam}), 
whose generic fiber may be the splitting field if $n=2$ of $\textbf{SO}_V$ (see in the proof of Theorem \ref{main theorem}).    
Therefore the Hasse principle holds for any unimodular form defined over it of any rank.    
So Theorem \ref{main theorem} is a generalization of Theorem~3.1 in \cite{Ger1} 
in which the elementary case of $\Oiy = \BF_q[t]$ is treated.   
\end{example}

\begin{remark}
The unimodularity condition is essential (though not necessary) for the validity of the Hasse principle even if $\Oiy$ is a UFD. 
It is necessary for the Clifford algebra construction if $n>2$, but is also required for rank $2$. 

For example, let $C$ be the projective line over $\BF_q$ and $\iy = (1/x)$. Then $\Oiy = \BF_q[x]$ is a UFD. 
Let $f$ and $g$ be the $\Oiy$-forms represented by  
$F = \diag((1-x^2)^2,1)$  and  $G = \diag((1-x)^2,(1+x)^2)$, respectively.  
Let 
$$ Q = \left( \begin{array}{cc}
 \fc{1}{1+x}  & 0  \\ 
         0    & 1+x  \\ 
\end{array}\right) \in \textbf{GL}_2(\hat{\CO}_\fp) \ \ \forall \fp \neq (1+x) $$ 
and  
$$ P = \left( \begin{array}{cc}
 0   & \fc{1}{1-x}   \\ 
 1-x & 0  \\ 
\end{array}\right) \ \in \textbf{GL}_2(\hat{\CO}_\fp) \ \ \forall \fp \neq (1-x). $$ 
Then $Q^t F Q = P^t F P = G$. 
This shows that $f$ and $g$ belong to the same genus.  
But they are not, however, isomorphic over $\Oiy$, 
since mapping the eigenvalue $(1-x^2)^2$ in $F$ to $(1-x)^2$ or $(1+x)^2$ in $G$    
can be done only by dividing by a non-constant element, which is not allowed in $\Oiy = \BF_q[x]$. 
\end{remark}

\begin{example} \label{EC}
Let $C$ be an elliptic $\BF_q$-curve and suppose $\iy$ is $\BF_q$-rational (such one must exist). 
The restriction of $C$ to $C^\af$ gives rise to an exact sequence (see \cite[Cha.II,~Prop.6.5(c)]{Hart}):
$$ 0 \to \Z \to \Pic(C) \to \Pic(C^\af) \to 0 $$
in which the first map $1 \mapsto 1 \cdot \{\iy\}$ is injective because the degree of a curve's divisor is well defined. 
As we assumed $\iy$ is $\BF_q$-rational, this sequence splits as abelian groups. 
The degree map on $\Pic(C)$ yields another exact sequence which again splits as abelian groups:
$$ 0 \to \Pic^0(C) \to \Pic(C) \to  \Z  \to 0.  $$ 
We get an isomorphism of summands $\Pic^0(C) \cong \Pic(C^\af)$.  
Together with another isomorphism of abelian groups: $C(\BF_q) \cong \Pic^0(C) ; P \mapsto [P]-[\iy]$ 
we may deduce that: 
\begin{equation} \label{iso with Pic}
C(\BF_q) \cong \Pic(C^\af). 
\end{equation}
Hence according to Theorem \ref{main theorem}, any unimodular form $f$ of rank $\geq 3$ defined over $\Sp \Oiy$ 
admits the Hasse principle if and only if there is no element of order $2$ in $C(\BF_q)$. 
For example, suppose $q>3$ and $\iy=(0:1:0) \in C(\BF_q)$ is removed, 
so the remaining (non-singular) affine curve $C^\af$ is given in affine coordinates by the Weierstrass form
$$ y^2 = x^3 + ax + b \ \ \text{for some} \ a,b \in \BF_q.  $$ 
Then $f$ admits the Hasse principle if and only if $C^\af$ does not have any $\BF_q$-point on the $x$-axis.  
\end{example}

\begin{cor} \label{c(f) in EC n geq 3}
Let $C$ be an elliptic $\BF_q$-curve and suppose $\iy \in C(\BF_q)$. 
Then for any integral unimodular form $f$ of any rank $n>2$ one has $c(f) = |C(\BF_q)/2|$.    
\end{cor}

\begin{lem} \label{number of classes in EC}
Let $C$ be an elliptic $\BF_q$-curve. 
Suppose that $-1 \in (\BF_q^\times)^2$ and $\iy \in C(\BF_q)$. 
Then $c(1_2)=|C(\BF_q)|$.   
\end{lem}

\begin{proof}
The orthogonal group scheme over $\Sp \Oiy$ of $1_2$ is $\un{\textbf{O}}_2$. 
Consider the exact sequence of smooth $\Oiy$-schemes (recall that $\text{char}(K)$ is odd): 
$$  1 \to \un{\textbf{SO}}_2 \to \un{\textbf{O}}_2 \to \un{\mu}_2 \to 1. $$ 
As $-1 \in (\BF_q^\times)^2$, the one dimensional torus $\un{\textbf{SO}}_2$ is split, 
and so $H^1_\et(\Oiy,\un{\textbf{SO}}_2) = \Pic(C^\af)$. 
Then due to isomorphism (\ref{iso with Pic}), 
$C(\BF_q) \cong H^1_\et(\Oiy,\un{\textbf{SO}}_2)$, 
classifying according to Proposition \ref{class number and cohomology} the integral forms in $\text{gen}(1_2)$. 
According to Hilbert 90 Theorem, 
this set is also equal to $\ker[H^1_\et(\Oiy,\un{\textbf{SO}}_2) \to H^1(K,\textbf{SO}_2)]$,  
so the geometric interpretation of this result   
is that non-trivial principal $\un{\textbf{SO}}_2$-bundles,  
which are in this case non-trivial line bundles of $C^\af$, 
become trivial while tensoring with $K$. 
This causes the failure of the Hasse principle.   
\end{proof}

\begin{remark} 
Lemma \ref{number of classes in EC} with isomorphism (\ref{iso with Pic}) show that the UFD condition for $\Oiy$  
is essential for the validity of the Hasse principle in case of rank $2$, even for unimodular forms as $1_2$. 
Moreover, even if $\Oiy$ is a UFD, it is still essential to assume for $n=2$ 
that the integer closure of $\Oiy$ in the splitting field of $\textbf{SO}_V$ is a UFD as well. 

For example, the elliptic curve $C := \{ ZY^2 = X^3 - XZ^2 - Z^3 \}$ defined over $\BF_3$ (in which $-1$ is not a square)    
has a single $\BF_3$-rational point $(0:1:0)$. 
Suppose we choose it to be $\iy$.  
Then $C^\af = \{ y^2=x^3-x-1 \}$ and $\Oiy = \BF_3(C^\af)$ is a UFD (by (\ref{iso with Pic})). 
But the extension $\Oiy'$ by $i = \sqrt{-1}$, 
being the integer closure of $\Oiy$ in the splitting field of $\textbf{SO}_2$, 
gives rise to more rational points of $C$ like $(-1:i:1)$. 
Thus as $\Pic(\Oiy)=0$, \'etale cohomology applied on sequence (\ref{norm torus sequence}) 
implies the bijection of the non-trivial sets: $H^1_\et(\Oiy,\un{\textbf{SO}}_2) \cong \Pic(\Oiy')$ 
(see in the proof of Theorem \ref{main theorem}),    
which shows according to Corollary \ref{Hasse} that the Hasse principle fails for the $\Oiy$-form $1_2$.  
\end{remark}

\begin{example} \label{EC explicit example}
Let $C = \{ Y^2Z = X^3 + XZ^2\}$ defined over $\BF_5$. 
Then: 
$$ C(\BF_5) = \{ (0:0:1), (1:0:2), (1:0:3), (0:1:0) \} \cong \Z/4\Z. $$
Removing $\iy = (0:1:0)$, we get the affine elliptic curve: 
$$ C^\af = \{y^2 = x^3 + x \} \ \ \text{with:} \ \ \Oiy = \BF_5[x,y]/(y^2-x^3-x). $$
According to Lemma \ref{number of classes in EC}, 
we have four $\Oiy$-isomorphism classes in $\text{gen}(1_2)$. 
The key obstruction here for finding explicit integral forms from the same genus of $1_2$ 
which are not $\Oiy$-isomorphic to it, 
is using the fact that $\Oiy$ is not a UFD in such a way 
that there exist distinct isomorphisms to $1_2$, defined over integer rings at distinct places.  

Explicitly, the affine supports of the points in $C(\BF_5)$ are: 
$$ \{ (0,0),  (1/2,0) = (3,0),  (1/3,0) = (2,0), \iy \}.  $$
Each one of the first three points corresponds to an intersection between $(y)$ and another prime: 
$$ (0,0) \leftrightarrow (y) \cap (x), \ (3,0) \leftrightarrow (y) \cap (x-3), \ (2,0) \leftrightarrow (y) \cap (x-2) $$
while $\iy$ is associated to the all affine curve $C^\af$. 
For any $t \in \{1,x,x+2,x-2\}$, the matrix 
$$ P_{(t)} =   
\left( \begin{array}{cc}
 1 + 3t  &  2 - t \\ 
 3 + t   &  1 + 3t  \\ 
\end{array}\right) $$ 
with $\det(P_{(t)}) = 2t$ is invertible in $C^\af - \{ t\}$ (in all $C^\af$ if $t=1$),   
giving rise to the integral form represented by
$$ G_t = P_{(t)}^tP_{(t)} =
\left( \begin{array}{cc}
 2t  & 0  \\
 0  & 2t  \\ 
\end{array}\right).$$
On the other hand, using the relation (which is due to the fact that $-1 \in (\BF_5^\times)^2$): 
$$ y^2 = x(x+2)(x-2) $$
one may define the matrices: 
$$Q_{(x)} =  \fc{y}{(x-2)(x+2)} \left( \begin{array}{cc}
 x-1     & x+1  \\  
 -(x+1)  & x-1  \\ 
\end{array}\right), \ 
Q_{(x-2)} =  \fc{y}{x(x+2)} \left( \begin{array}{cc}
 x-1  & x+3  \\  
 x+3  & -(x-1)  \\ 
\end{array}\right) $$
$$ 
Q_{(x+2)} =  \fc{y}{x(x-2)} \left( \begin{array}{cc}
 x+1 & x+2  \\  
 x+2 & -(x+1)  \\ 
\end{array}\right)
$$ 
satisfying each $Q_{(t)}^tQ_{(t)}=G_t$ as well, and being invertible at the remaining place $(t)$. 

\bk

We get four non-equivalent $1$-cocycles, since $\sqrt{\det(G_{t_1})/\det(G_{t_2})} = t_1/t_2$ 
is not invertible over $\Oiy$ 
for any $t_1 \neq t_2$. 
The generic fibers of these cocycles are trivial, since both $P_{(t)}$ and $Q_{(t)}$ are well defined over $K$ for any $t$,  
and the transition maps $Q_{(t)}^t Q_{(t)} \cdot (P_{(t)}^t P_{(t)})^{-1}$ are trivial.   
The four corresponding non-isomorphic integral forms in $\text{gen}(1_2)$ are those represented by $\{G_t\}$. 

\bk

For $n>2$, however, any unimodular form $f$ defined over this domain $\BF_5[C^\af]$ 
will admit according to Corollary \ref{c(f) in EC n geq 3} only $|C(\BF_q)/2| = 2$ classes in the genus of $f$. 
\end{example}


\bk

\end{document}